\documentclass[12pt]{amsart}

\usepackage{amssymb,enumerate,bookmark}

\numberwithin{equation}{section}

\DeclareMathOperator{\vol}{Vol}%
\DeclareMathOperator{\const}{const}%

\newtheorem{theorem}{Theorem}

\newtheorem{Lemma}{Lemma} 
\newtheorem{corollary}{Corollary} 

\newtheorem*{question*}{Question}
\newtheorem{proposition}{Proposition}
\newtheorem*{problem}{Problem}

\theoremstyle{definition}

\theoremstyle{remark}
\newtheorem*{remark*}{Remark}
\newtheorem{example}{Example}

\newcommand{\weg}[1]{}
\newcommand{\be}{\begin{equation}}
\newcommand{\ee}{\end{equation}}

\newcommand \n{\nabla}

\newcommand \ve{\varepsilon}
\newcommand \id{\mathrm{id}}

\newcommand \Rn{\mathbb R^n}

\newcommand \End{\operatorname{End}}

\newcommand \SO{\mathrm{SO}}

\title[Locally conformally Berwald manifolds ]{Locally conformally Berwald manifolds and compact quotients of reducible
manifolds by homotheties} 
\date{\today}
\author{Vladimir S. Matveev}
\address{Institute of Mathematics, Friedrich-Schiller-Universit\"at Jena,  07737 Jena Germany}
\email{vladimir.matveev@uni-jena.de}
\thanks{The first author was partially supported by DFG (GK 1523), DAAD  and FSU Jena}
\author{Yuri Nikolayevsky}
\address{Department of Mathematics and Statistics, La Trobe University, Melbourne, Victoria, 3086, Australia}
\email{y.nikolayevsky@latrobe.edu.au}
\thanks{The second author was partially supported by ARC Discovery grant DP130103485}

\begin{document}
\keywords {Finsler manifold, Berwald manifold, homothety group, reducible holonomy}
\subjclass[2010]{53C60, 53C22, 53B40, 53C29}

\begin{abstract}
We study locally conformally Berwald metrics on closed manifolds which are not globally conformally Berwald. We prove that the characterization of such metrics is equivalent to characterizing incomplete, simply-connected, Riemannian manifolds with reducible holonomy group whose quotient by a group of homotheties is closed. We further prove a de Rham type splitting theorem which states that if such a manifold is analytic, it is isometric to the Riemannian product of a Euclidean space and an incomplete manifold.
\end{abstract}

\maketitle

\section{Introduction} 
\label{s:intro}

A \emph{Finsler metric} on a manifold $M$ of dimension $n\geq 2$ is  a continuous function $F:TM\to [0,\infty)$ that is smooth on the slit tangent bundle $TM^0 = TM \setminus (\text{the zero section})$ and such that for every point $x\in M$ the restriction $F_x := F_{|T_xM}$ is a \emph{Minkowski norm}, that is, $F_x$ is positive homogenous and convex and it vanishes only at $v = 0$:
\begin{enumerate}[(a)]
 \item $F_x(\lambda \cdot v) = \lambda \cdot  F_x (v) $ for any $\lambda \geq 0$.
  \item $F_x (v + u) \le F_{x} (v) + F_{x} (u)$.
  \item $F_x (v)= 0 $ \ $ \Rightarrow$ \ $v=0$.
\end{enumerate}
We do not require that the metric is \emph{reversible}, i.e., that $F_x(v) = F_x(-v)$, and that it is strictly convex, i.e., that the second differential $d^2\left((F_x)^2\right)$ is positive definite.

We will assume all the objects in this paper to be at least as smooth as we need for the proofs, and all the manifolds under consideration to be connected.

\vspace{1ex}
We say that a Finsler metric $F$ is \emph{Berwald}, if there exists a torsion free affine connection $\nabla$ on $M$ whose parallel transport preserves $F$: if $\gamma$ is a smooth path in $M$ with the endpoints $x$ and $y$, and $P_{\gamma}: T_xM \to T_yM$ is the $\nabla$-parallel transport along $\gamma$, then
\begin{equation} \label{par1}
F_y(P_{\gamma}(v)) = F_x(v)
\end{equation}
for all $v \in T_xM.$


A Riemannian metric $g= g_{ij}$ viewed as a Finsler metrics with $F_x(v)= \sqrt{g_{ij}(x)v^iv^j}$ is Berwald, with the associated connection $\nabla$ the Levi-Civita connection. A simple example of a Berwald non-Riemannian metric is  a \emph{Minkowski} metric on $\Rn$  obtained by the following procedure: take a Minkowski norm $F_0$ on $\Rn$ and define the Finsler metric on $\Rn $ by setting $F_x(v)= F_0(v)$, for all $x \in \Rn$. The resulting metric is clearly a Berwald metric whose associated connection is the flat connection on $\Rn$.

We say that a Finsler metric $F$ is \emph{locally conformally Berwald}, if for every point $x \in M$ there exist a neighborhood $U(x)$ and a positive function $\lambda:U(x)\to \mathbb{R}$ such that the conformally related Finsler metric $\lambda F$ is Berwald. We say that a Finsler metric $F$ is \emph{globally conformally Berwald}, if such a function $\lambda$ exists on the whole manifold.

\vspace{1ex}
In this paper we study locally conformally Berwald closed manifolds which are not globally conformally Berwald. The simplest example of such a manifold is given below; more complicated and interesting examples can be constructed from the Riemannian metrics in \cite{MN} using the approach from Section~\ref{s:th1}.

\begin{example} \label{ex1} Let $F$ be an arbitrary Minkowski metric on $\mathbb{R}^n$. Consider the mapping
\begin{equation*}
 \alpha: \mathbb{R}^n \setminus \{ 0\} \to \mathbb{R}^n \setminus \{ 0\}, \ \  x \mapsto q x,
\end{equation*}
where $q\ne 1$ is positive. This mapping generates a free, discrete action of the group $\mathbb{Z}$ on $\mathbb{R}^n \setminus \{0\}$. The quotient space $M= (\mathbb{R}^n \setminus \{0\})/ \mathbb{Z}$ is diffeomorphic to $S^{n-1}\times S^1$. Since the group $\mathbb{Z}$ acts by isometries of the metric $\tfrac{1}{\|x\|} F$, the metric $\tfrac{1}{\|x\|} F$ induces a Finsler metric on $M$. Since the metric $\tfrac{1}{\|x\|} F$ is conformally related to the Berwald (even Minkowski) metric $F$, the induced metric on $M$ is locally conformally Berwald. The Minkowsky metric  $F$ restricted to  $\mathbb{R}^n \setminus \{ 0\}$ is clearly not complete, which, as we explain later, implies that the quotient is not a globally conformally Berwald manifold,  see  Corollary~\ref{cor:7}.
\end{example}

Our first result, Theorem~\ref{thm:reduction} below, reduces the study of locally, but not globally, conformally Berwald closed manifolds to the following purely Riemannian problem.

\begin{problem}
Characterize closed quotients of simply-connected incomplete Riemannian manifolds with reducible holonomy group by a free action of a group of homotheties.
\end{problem}

A diffeomorphism $\phi$ of a Riemannian manifold $(\tilde M, g)$ is called a \emph{homothety}, if the pullback $\phi^*g$ is a constant multiple of $g$; if $\phi^*g = g$, the homothety is an \emph{isometry}. 

\begin{theorem} \label{thm:reduction}
Let $M$ be a smooth manifold with the universal cover $\tilde M$ and let $G=\pi_1(M)$. The following two properties are equivalent.

\begin{enumerate}[{\rm (1)}]
 \item \label{it:red1}
 $M$ admits a locally conformally Berwald Finsler metric which is not globally conformally Berwald.

 \item \label{it:red2}
 $\tilde M$ admits an incomplete Riemannian metric $g$ with reducible holonomy group such that $G$ acts on $\tilde M$ by homotheties.
\end{enumerate}
\end{theorem}

It is relatively easy to see and will be explained in the proof of Theorem~\ref{thm:reduction} that if $G$ is a discrete, cocompact group of homotheties freely acting on a \emph{complete} manifold, then $G$ consists only of isometries. So the assumption of incompleteness in Theorem~\ref{thm:reduction}\eqref{it:red2} can be replaced by requiring that the group $G$ contains not only isometries, which is sometimes termed by saying that the homothety group $G$ is \emph{essential}.

The proof of Theorem~\ref{thm:reduction} which we give in Section~\ref{s:th1} is constructive in the both directions. Note that our construction of a locally reducible, incomplete metric $g$ on $\tilde M$ from a locally, but not globally, conformally Berwald metric $F$ produces a unique metric, up to a constant multiple. Other constructions are possible, but the resulting metrics on $\tilde M$ are somehow ``related" and depend on a finite number of constants. Meanwhile, the construction in the opposite direction depends on a more or less arbitrary choice of a norm on a finite-dimensional space (see Example~\ref{Ex3}); so, although the resulting locally conformally Berwald metrics are related -- for example they have the same canonical connection (Corollary~\ref{cor:existcan}) -- but still there is an infinite-dimensional freedom in choosing such a metric.

\vspace{1ex}
The above Problem was studied in the literature, see e.g. the recent paper \cite{BM} and references therein. Actually, it was believed (see \cite[Conjecture~1.3]{BM}) that if a quotient of an incomplete, simply-connected, Riemannian manifold $(\tilde M, g)$ by a free, discrete action of a group $G$ of homotheties is compact, then $g$ is flat. Would this conjecture be true, a complete description of locally, but not globally conformally Berwald manifolds, would follow from \cite{Fried}, see \cite[\S 5]{Troyanov} for details. Unfortunately, a counterexample to this conjecture was constructed in \cite{MN}. Our second result states that certain phenomena observed in that counterexample are present in every (real-analytic) metric of this class -- we prove the following
de Rham type splitting theorem.

\begin{theorem} \label{thm:main} Let $(\tilde M, g)$ be an incomplete real-analytic simply-connected Riemannian manifold whose holonomy group is reducible. Let $G$ be a freely, discretely acting group of  homotheties such that the quotient $M=\tilde M/G$ is compact. Assume the metric $g$ is not flat. Then $(\tilde M, g)$ is isometric to the direct product $(\mathbb{R}^k, g_{\mathrm{standard}})\times (N,h)$, for some $k \ge 1$ and some Riemannian manifold $(N, h)$ \emph{(}which is automatically neither complete nor flat\emph{)}. 
\end{theorem}


\section{Proof of Theorem \ref{thm:reduction}}
\label{s:th1}

\subsection{Plan of proof}
\label{subsecplan}
In Section~\ref{subsec1} we recall the definition and the main properties of the Binet-Legendre metric and the description of Berwald metrics using the Binet-Legendre metric. That description will immediately produce the main construction in the proof of implication \eqref{it:red2} $\Longrightarrow$ \eqref{it:red1} of Theorem~\ref{thm:reduction}, though to prove that the resulting metric is indeed locally, but not globally Berwald, we will need results of the following sections.

In Section~\ref{subsec2} we prove, for conformally Berwald non-Riemannian Finsler metrics, the existence of a unique affine torsion-free connection, the \emph{canonical} connection, whose parallel transport preserves any Berwald metric in the conformal class. This will follow from the results of Cs. Vincze \cite{vincze1, vincze2}; we will give a new, shorter proof using the Binet-Legendre metric.

In Section~\ref{subsec3} we show that if our manifold is closed, the canonical connection has reducible holonomy group. We also show that the connection is complete if and only if metric is globally conformally Berwald. This will give us all the ingredients for the proof of Theorem~\ref{thm:reduction} in Section~\ref{ss:pfth1}.

\subsection{ Binet-Legendre Metric and its properties for Berwald non-Riemannian metrics} \label{subsec1}
One of Riemannian metrics which can be constructed from a given Finsler metric $F$ is the \emph{Binet-Legendre metric} $g_F$. The construction goes as follows. For every $x \in M$, consider the convex set $K_x:= \{ v \in T_xM \mid F_x(v)\le 1\}$ and choose an arbitrary (linear) volume form $\Omega$ on $T_xM$. Introduce an inner product $g_x^*$ on $T^*_xM$ by setting
\begin{equation*}
  g_x^*(\xi, \eta) := {\frac{n+2}{\vol_\Omega(K_x)} \int_{K_x} \xi(v) \eta(v) d\Omega}
\end{equation*}
for two linear forms $\xi, \eta \in T^*_xM$.

It is not difficult to show (or see \cite{Troyanov} for details) that
\begin{itemize}
\item $g_x^*$ is bilinear, symmetric and positive definite;
\item $g_x^*$ does not depend on the choice of the (linear) volume form $\Omega_x$;
\item $g_x^*$ smoothly depends on $x$, if $F$ is smooth (or even \emph{partially smooth}, see \cite{Troyanov});
\item $g_x^*$ behaves as a $(2,0)$-tensor under coordinate changes.
\end{itemize}
The \emph{Binet-Legendre} metric $g_F$ associated to the Finsler metric $F$ is the Riemannian metric dual to $g^*$. The construction first appeared in \cite{Centore}; it was then rediscovered in \cite{Troyanov} and has been used there to solve several named problems.

The Binet-Legendre metric is a useful tool in the study of Berwald Finsler metrics; for example, with its help one can give shorter proofs of some of the classical results of Szab\'o \cite{Sz1} and Vincze \cite{vincze0}. We start by observing that for a Berwald metric $F$, its associated connection $\nabla$ is the Levi-Civita connection of the Binet-Legendre metric $g_F$. Indeed, let $\gamma$ be a smooth path with the endpoint $x , y \in M$ and let $P_{\gamma}: T_xM \to T_yM$ be the $\nabla$-parallel transport along $\gamma$. Then \eqref{par1} implies $P_{\gamma}(K_x)= K_y$. Since the construction of the Binet-Legendre metric only requires the linear structure and the set $K$, the linear map $P_\gamma$ sends the Binet-Legendre metric at $x$ to the Binet-Legendre metric at $y$. So $g_F$ is parallel relative to $\nabla$ implying that $\nabla$ is the Levi-Civita connection of $g_F$.

Consider now the \emph{restricted holonomy group} $\mathrm{Hol}^0_x(\nabla)$, the subgroup of $\End(T_xM)$ generated by parallel transports $P_\gamma$ along contractible loops $\gamma$ starting (and ending) at $x$. It is a subgroup of the (full) \emph{holonomy group} $\mathrm{Hol}_x(\nabla)$ which is generated by parallel transports $P_\gamma$ along all (not necessarily contractible) loops $\gamma$ starting and ending at $x$. Clearly, if $\n$ is the associated connection of a Berwald metric $F$, then $\mathrm{Hol}^0_x(\nabla)$ preserves $F_x$ and therefore $g_F(x)$, and hence $\mathrm{Hol}^0_x(\nabla)\subseteq \SO(T_xM, g_F(x))$.

Suppose that $\mathrm{Hol}^0_x(\nabla)$ acts transitively on the unit $g_F$-sphere in $T_xM$. Then the ratio $\tfrac{\left(F_x(v)\right)^2}{g_F(v,v)}$ does not depend on the point $v$ on the sphere. Since it is homogeneous of order zero, it is constant on the whole slit tangent space $T_xM\setminus \{0\}$. Then the restriction of $F$ to $T_xM$ is a constant multiple of $\sqrt{g(v,v)}$, which implies that $F$ is a Riemannian metric.

If the restricted holonomy group does not act transitively on the unit sphere, then by the classical result of Berger \cite{berger} and Simons \cite{simons} either $\mathrm{Hol}^0_x(\nabla)$ is reducible (which implies that locally $g_F$ is a direct product), or $g_F$ is the metric of a locally symmetric space of rank greater than $1$.

Note also that if the restricted holonomy group of a non-flat Riemannian metric is reducible, then the full holonomy group is not transitive on the unit sphere. Indeed, in this case there exists a orthogonal decomposition $T_xM = V_0 \oplus V_1 \oplus \ldots \oplus V_m$ of the tangent space in the direct sum of invariant subspaces (we assume that the action on $V_0$ is trivial and the decomposition is maximal). The action of the holonomy group may permute the invariant subspaces (with some restrictions; for example, the permuted subspaces must have the same dimension and the ``flat'' subspace $V_0$ remains stable), but can not change the decomposition. 

Now, if the holonomy group of the Levi-Civita connection $\nabla$ of a Riemannian metric $g$ is not transitive on the unit sphere in $T_xM$, then there exists a Berwald non-Riemannian Finsler metric whose associated connection is $\nabla$. Indeed, take a $\mathrm{Hol}_x$-invariant norm $F_x$ on $T_xM$ and extend it to the tangent space $T_yM$ at an arbitrary point $y \in M$ by a parallel transport $P_\gamma$ along a curve $\gamma$ connecting $x$ and $y$. Since the norm is invariant with the respect to the holonomy group, the construction does not depend on the choice of $\gamma$. Moreover, the resulting metric is smooth and is preserved by $\nabla$-parallel transport, that is, is a Berwald metric. Of course, for this construction to work, one needs to explain why there exists a $\mathrm{Hol}_x$-invariant norm. We will not need such explanation in the case of locally symmetric metrics $g$ (which was the hard part of the prove of \cite[Theorem~1]{Sz1}); let us only mention that many such norms exist and that their construction can be easily done using Chevalley's polynomials which are preserved by the holonomy group, see e.g. \cite{Planche1, Planche2}. We will however need the construction of the $\mathrm{Hol}_x$-invariant norm in the case when the restricted holonomy group is reducible (similar to the construction in \cite{Sz1}); we start with a simple local example, whose easy generalization will be used in the proof of Theorem \ref{thm:reduction}.

\begin{example} \label{ex2} Consider the direct product $(M_1 \times M_2, g_1 + g_2)$ of two Riemannian manifolds $(M_1, g_1)$ and $(M_2, g_2)$. The tangent space at a point $(x_1, x_2)\in M_1 \times M_2$ naturally splits into the direct sum of $T_{x_1}M_1$ and $T_{x_2}M_2$.

Given a reversible Minkowski norm $N$ on $\mathbb{R}^2$, define the Finsler metric $F$ by
\begin{equation*}
F_{(x_1,x_2)}(v_1 + v_2) = N(\|v_1\|, \|v_2\|),
\end{equation*}
where $\|v_i\|$ is the $g_i$-norm of a vector $v_i$ tangent to $M_i$. Since the parallel transport in the Levi-Civita connection of $g$ preserves the metrics $g_1$ and $g_2$ and the property of a vector to be tangent to $M_i$, it also preserves the Finsler metric $F$; hence $F$ is Berwald.
\end{example}

Example~\ref{ex2} can be easily generalized to the direct product of $k$ Riemannian manifolds. If $k=1$, we obtain Riemannian metrics (which are, of course, Berwald), and if $k=n$, we obtain, at least locally, all the (reversible) Minkowski metrics. One can also slightly modify Example~\ref{ex2} to obtain, locally, all the Minkowski metrics. Indeed, if a metric $g_i$ on $M_i$ is flat, then we do not need reversibility of $N$ with respect to the $i$-th coordinate.

The construction in the following example essentially proves the implication \eqref{it:red2} $\Longrightarrow$ \eqref{it:red1} of Theorem~\ref{thm:reduction}.

\begin{example} \label{Ex3}
Suppose $(\tilde M, g)$ is an incomplete, simply-connected Riemannian manifold whose holonomy group is reducible. For any $x \in \tilde M$ we have a decomposition
\begin{equation} \label{decomposition}
T_x \tilde M= V_0\oplus V_1\oplus \ldots \oplus V_m,
\end{equation}
where the subspaces $V_i$ are $\mathrm{Hol}_x$-invariant, the action of $\mathrm{Hol}_x$ on $V_0$ is trivial ($V_0$ may have dimension $0$) and the action on each of the components $V_1, \ldots, V_m$ is irreducible.

Choose a reversible, non-Euclidean norm $N$ on $\mathbb{R}^{m+1}$, and use it to construct a norm $F_x$ on $T_x \tilde M$ by setting
\begin{equation*}
 F_x(v_0 + v_1+ \ldots+ v_m)= N(\|v_0\|,\ldots, \|v_m\|),
\end{equation*}
were $v_i\in V_i$ and $\|v_i\|$ is the $g$-norm. The norm $F_x$ is $\mathrm{Hol}_x$-invariant  and therefore induces a Berwald Finsler metric $F$ on $\tilde M$, whose associated connection is the Levi-Civita connection of $g$.

Now suppose a group $G$ acts freely and discretely on $\tilde M$ by homotheties of $g$ and let $M=\tilde{M}/G$. The differential of an element $\phi \in G$ respects the splitting \eqref{decomposition} in the following sense: it sends $V_0(x)$ to $V_0(\phi(x))$, and sends every $V_i(x), \; i > 0$, to some $V_j(\phi(x)), \; j > 0$, of the same dimension. We now impose the following additional assumption on $N$: for any $i, j > 0$ such that $\dim V_i = \dim V_j$, we require $N$ to be invariant with respect to the interchanging of the $(i+1)$-st and $(j+1)$-st coordinates: $N(\dots,\|v_i\|,\dots,\|v_j\|,\dots)= N(\dots,\|v_j\|,\dots,\|v_i\|,\dots)$. This guarantees that any $\phi \in G$ is also a homothety of $F$. Then we can find a smooth positive function $f$ on $\tilde M$ such that $G$ acts by isometries of the Finsler metric $F'=fF$. The projection of $F'$ to $M$ is then a locally conformally Berwald metric. Moreover, since the associated connection of $F$ is not complete, that metric is not globally conformally Berwald; this will follow from Theorem~\ref{th2} below.
\end{example}

\subsection{Canonical connection for a conformally Berwald non-Riemannian metric}
\label{subsec2}

We start with giving a new proof of the following well-known result; our proof is more Riemannian and requires no Finsler technique.

\begin{theorem}[Cs. Vincze, \cite{vincze1,vincze2}] \label{Th2}
Let $F$ be a Berwald metric on a connected manifold of dimension $n\ge 2$. Assume that a conformally related metric $\tilde F=e^{\varphi} F$ is also Berwald. Then either $F$ is a Riemannian metric, or $\varphi$ is a constant.
\end{theorem}

\begin{proof}
The Binet-Legendre metrics $g:= g_{F}$ and $\tilde g:= g_{\tilde F}$ are related by $\tilde g= e^{2\varphi} g$. The conformal class of the both metrics $F$ and $\tilde F$ is preserved by parallel transports for Levi-Civita connections of both $g$ and $\tilde g$. Take a point $x$, a tangent vector $v\in T_xM$ and a smooth curve $\gamma(t)$, where $t \in [0,\varepsilon]$ is an arclength parameter, with $\gamma(0)=x, \, \dot\gamma(0)=v$. Consider the composition of the parallel transport along $\gamma$ from $x$ to $\gamma(\ve)$ using the Levi-Civita connection of $g$ with the parallel transport along the same curve $\gamma$, but in the opposite direction, from $\gamma(\ve)$ to $x$, and using the Levi-Civita connection of $\tilde g$. The resulting endomorphism $\Phi_\ve$ of $T_xM$ preserves the conformal class of $F_x$, hence multiplies $F_x$ by a constant. On the other hand, for small $\varepsilon$ we have $\Phi_\ve = \id + \ve L + o(\ve)$, where the transformation $L$ of $T_xM$ is generated by the difference of connections, which for conformally related metrics is given by 
\begin{equation} \label{prev}
  \tilde \Gamma^k{}_{ij} -\Gamma^k{}_{ij}= \delta^k_i\partial_j\varphi + \delta^k_j\partial_i\varphi-g_{ij}\nabla^k\varphi,
\end{equation}
and hence is represented by the matrix
\begin{equation*}
L^k_i=v^j(\tilde \Gamma^k{}_{ij} -\Gamma^k{}_{ij}) = v^j (\delta^k_i \varphi_{,j} + \delta^k_j \varphi_{,i}-g_{ij}\varphi_{, }^{\ k})=
  d\varphi(v) \delta^k_i + v^k \varphi_{,i} - v_i \varphi_{,}^{\ k},
\end{equation*}
which must lie in the Lie algebra of the Lie group of linear transformations preserving the conformal class of $F_x$. Thus, for any $v\in T_xM$ the Finsler metric $F_x$ is invariant with respect to the rotation generated by the skew-symmetric matrix $v^k \varphi_{,j} - v_j \varphi_{,}^{\ k}$. It follows that $F_x$ is rotationally-symmetric with respect to the group $\SO(n-1)$ of rotations around $\varphi_{,}^{\ k}= \operatorname{grad}(\varphi)$.

We see that if $F$ is not Riemannian and $d\varphi\ne 0$, then the direction of the vector field $\operatorname{grad}(\varphi)$ is uniquely determined by $F$.
Hence, the $g$-unit vector field in that direction is parallel relative to the Levi-Civita connection of $g$. Then the distribution $D = \ker d\varphi$ orthogonal to $\operatorname{grad}(\varphi)$, is integrable and totally geodesic, hence for any vector field $v\in D$ we have $\nabla_v v \in D$. Swapping the metrics $g$ and $ \tilde g$, we obtain, by a similar argument, that $\tilde \nabla_v v \in D$. So for any vector field $v\in D$ we obtain
\begin{equation*}
\tilde \nabla_v v - \nabla_v v \in D,
\end{equation*}
which implies in view of  \eqref{prev} that $\operatorname{grad}(\varphi)=0$. Finally, the metric $F$ is Riemannian, or $\varphi= \const$.
\end{proof}

\begin{corollary} \label{cor:existcan}
Let $(M, F)$ be a locally conformally Berwald non-Riemannian manifold. Then there exists an unique affine, torsion-free connection whose parallel transport preserves any Berwald metric in the conformal class.
\end{corollary}

We call this connection the \emph{canonical connection} on a locally conformally Berwald non-Riemannian manifold $(M, F)$.

\begin{proof} By Theorem~\ref{Th2}, all Berwald metrics in a (non-Riemannian) conformal class are proportional, so their Binet-Legendre metrics are proportional, and hence their Levi-Civita connections (which are the associated connections of the corresponding Berwald metrics) are equal.
\end{proof}

\begin{corollary}\label{cor:si}
A locally conformally Berwald metric on a simply-connected manifold $\tilde M$ is globally conformally Berwald.
\end{corollary}
\begin{proof}
If the metric is Riemannian, it is nothing to prove. If it is not, then by Theorem~\ref{Th2}, the set of conformally Berwald metrics in the conformal class can be identified with sections of a one-dimensional $(\mathbb{R}_{>0})$-bundle over the manifold; Berwald metrics in the conformal class correspond to parallel sections of a certain linear connection. Since by construction there exist local parallel sections, the existence of a global parallel section (i.e., of a Berwald metric in the conformal class) follows from the simply-connectedness of $\tilde M$.
\end{proof}

\begin{corollary} \label{cor:6} Let $(M, F)$ be a simply-connected conformally Berwald non-Riemannian metric, and $\tilde F$ be a Berwald metric in the conformal class ($\tilde F$ exists by Corollary~\ref{cor:si}). Then any conformal transformation of $F$, i.e., any diffeomorphism $\varphi:M \to M$ that sends $F$ to a conformally equivalent metric, is a homothety of $\tilde F$, and therefore is a homothety of its Binet-Legendre metric $g_{\tilde F}$.
\end{corollary}
\begin{proof}
The conformal transformation $\varphi$ maps $\tilde F$ to a Berwald metric which is conformally equivalent to $F$. By Theorem \ref{Th2}, it is a constant multiple of $\tilde F$.
\end{proof}

\subsection{Incompleteness and reducibility of locally conformally Berwald metric on a closed manifold}
\label{subsec3}

Suppose now the manifold $M$ is closed (compact, with no boundary).

\begin{theorem} \label{th2}
Let $(M, F)$ be a closed, locally conformally Berwald non-Riemannian manifold. The canonical connection is complete if and only if the manifold is globally conformally Berwald.
\end{theorem}
\begin{proof} The sufficiency is obvious: if the metric $F$ is globally conformally Berwald, then the associated connection is the Levi-Civita connection of a Riemannian metric on a closed manifold and is therefore complete.

To prove the necessity we show that if the canonical connection is complete then the manifold is globally conformally Berwald. Consider the universal cover $\tilde M$ and the action of the fundamental group $\pi_1(M)$ on $\tilde M$ by deck transformations. By Corollary \ref{cor:si}, there exists a Berwald metric $\tilde F$ on $\tilde M$ in the conformal class of the lift of the metric $F$. By Corollary \ref{cor:6}, the action of the fundamental group $\pi_1(M)$ on $\tilde M$ is the action by homotheties of the Binet-Legendre metric $g_{\tilde F}$ on $\tilde M$.

By assumption, the lift of the canonical connection is complete. Then the Binet-Legendre metric $g_{\tilde F}$ is also complete. If there is an element of $\pi_1(M)$ which acts by a homothety $\phi$ with coefficient $k \in (0, 1)$, then $\phi$ has a fixed point, which contradicts the fact that the action of $\pi_1(M)$ is free.
Therefore $\pi_1(M)$ acts by isometries of $(\tilde M, g_{\tilde F})$ and of $(\tilde M, \tilde F)$. Then the metric $\tilde F$ projects to a Berwald metric on $M$. 
\end{proof}

\begin{corollary} \label{cor:7}
The locally conformally Berwald metric from Example~\ref{ex1} is not globally conformally Berwald.
\end{corollary}
\begin{proof}
The (lift of the) canonical connection is the standard flat connection on $\Rn$, which is clearly not complete, since we removed the point $0$.
\end{proof}

From Section~\ref{subsec1} we know that the Binet-Legendre metric of a non-Riemannian Berwald Finsler metric either has a reducible restricted holonomy group or is a locally symmetric Riemannian metric of rank at least 2. The next theorem shows that the second possibility can never happen in our settings.

\begin{theorem} \label{th3}
Let $(M, F)$ be a closed, locally conformally Berwald manifold, which is not globally conformally Berwald. Then the restricted holonomy group of the canonical connection is reducible.
\end{theorem}

\begin{proof} As $(M,F)$ is not globally conformally Berwald, it is non-Riemannian. Furtermore, the universal cover $(\tilde M, \tilde F)$ is globally conformally Berwald by Corollary~\ref{cor:si}, so there exists a Berwald metric $F'$ on $\tilde M$ conformally equivalent to $\tilde F$. If the restricted holonomy group of the canonical connection for $(M,F)$ is irreducible, then   the Binet-Legendre metric $g_{F'}$
  is a locally symmetric Riemannian metric. Then the squared norm $\|R_{F'}\|^2$ of the curvature tensor of $g_{F'}$ is constant on $\tilde M$ and is positive, as $g$ is not flat. If an element of $\pi_1(M)$ acts by a homothety $\phi$ with coefficient $k$, we get $\|R_{\phi^*F'}\|^2 = \|R_{k F'}\|^2 = k^{-2} \|R_{F'}\|^2$. Therefore $\pi_1(M)$ acts by isometries of $(\tilde M, F')$. Then the metric $F'$ projects to a Berwald metric on $M$ which is conformally equivalent to $F$.
\end{proof}

\subsection{Proof of Theorem \ref{thm:reduction}}
\label{ss:pfth1}
The implication \eqref{it:red2} $\Longrightarrow$ \eqref{it:red1} follows from the construction in Example~\ref{Ex3} and Theorem~\ref{th2}.

To prove the implication \eqref{it:red1} $\Longrightarrow$ \eqref{it:red2}, consider a locally, but not globally, Berwald non-Riemannian metric $F$ on $M$. By Corollary~\ref{cor:si}, there exists a Berwald metric $\tilde F$ on the universal cover $\tilde M$ such that $\tilde F$ lies in conformal class of the lift of $F$. We denote by $g$ the Binet-Legendre metric of $\tilde F$. By Corollary~\ref{cor:6}, the deck transformation action of the fundamental group $\pi_1(M)$ on $\tilde M$ is homothetic with respect to $g$. By Theorem~\ref{th2}, the metric $g$ is not complete. Finally, by Theorem~\ref{th3}, the metric has reducible restricted holonomy group.
\qed

\section{ Proof of Theorem \ref{thm:main}}
\label{s:proof2}

\subsection{Plan of proof} 
\label{s:setup}

Let $(\tilde{M}, g)$ be a simply-connected, connected, incomplete Riemannian manifold whose holonomy group is reducible, and let $G$ be a group of homotheties of $(\tilde{M}, g)$ which acts freely, discretely and in such a way that the quotient $M=\tilde{M}/G$ is a closed manifold; we denote by $\pi: \tilde{M}\to M=\tilde{M}/G$ the natural projection.

We want to show that $(\tilde{M}, g)$ is isometric to the direct product $(\mathbb{R}^k, g_{\mathrm{standard}}) \times (N, h)$. As we will see, the assumption of real analyticity of the metric which we imposed in the theorem, is not necessary until the very end of the proof (and we hope that the conclusion remains true without it -- see the Remark in the end of Section~\ref{ss:further}). 

\smallskip

Since the holonomy group is reducible, $(\tilde{M}, g)$ carries two orthogonal, totally geodesic, $G$-ivariant foliations $\mathcal{F}_1$ and $\mathcal{F}_2$ (of positive dimension; we do not assume that the foliations are irreducible. Without loss of generality we assume that $G$ preserves the foliations; one can always achieve that by passing to a subgroup of $G$ of finite index). We denote by $\mathcal{F}_i(x)$ the leaf of $\mathcal{F}_i, \; i=1,2$,  passing through $x \in \tilde M$. Let $g_i, \; i=1,2$, be the restrictions of the metric $g$ to the leaves of foliations $\mathcal{F}_i, \; i=1,2$, and $R_i, \; i=1,2$, the squared norms of the curvature tensors of $g_i$, respectively.

Let $d$ be the distance function on $(\tilde{M}, g)$ and $\overline{M}$ the metric completion of $(\tilde{M},g)$. Denote $M_\infty= \overline{M}\setminus \tilde{M}$ and consider the positive function
\begin{equation*}
d_\infty: \tilde{M}\to \mathbb{R}_{> 0}, \quad x \mapsto \inf\{d(x,y)\mid y\in M_\infty\}.
\end{equation*}
By the triangle inequality, the function $d_\infty$ is continuous.

For $x \in \tilde M$, let $p\in M_\infty$ be such that $d(x,p)=d_\infty(x)$ (in $\overline{M}$), and let $\gamma:[0, d_\infty(x)] \to \overline{M}$ be a continuous map such that the restriction $\gamma_{[0, d_\infty(x))}$ is a geodesic of $\tilde M$ parametrized by the arclength. We call $\gamma$ a \emph{minimal $g$-geodesic connecting $x$ with $p$}, or simply a \emph{minimal geodesic starting at $x$}.

The proof is based on the following two propositions.

\begin{proposition} \label{p:Ri=0}
At every point $x \in \tilde{M}$ we have either $R_1(x)= 0$ or $R_2(x)=0$ \emph{(}or both\emph{)}. More precisely, if there exists a minimal geodesic $\gamma$ starting at $x$, which does not lie in the leaf of $\mathcal{F}_i$ passing through $x$, then $R_i = 0$ in a neighborhood of $x$ in $\tilde M$.
\end{proposition}

\begin{proposition}\label{l:complete}
Suppose a minimal geodesic $\gamma$ starting at $x\in \tilde M$ lies on the leaf $\mathcal{F}_2(x)$ \emph{(}respectively, $\mathcal{F}_1(x)$\emph{)}. Then the leaf $\mathcal{F}_1(x)$ \emph{(}respectively, $\mathcal{F}_2(x)$\emph{)} is complete and flat and the restriction of $d_\infty$ to it is a constant.
\end{proposition}

Once these two propositions are proved, the theorem easily follows. Indeed, by the assumption, the metric $g$ is not flat. Then there is a point $x \in \tilde M$ where one of the curvatures, say $R_2(x)$, is nonzero, which by Proposition~\ref{p:Ri=0} implies that for any $y$ close to $x$, any minimal geodesic $\gamma$ starting at $y$ lies on the leaf $\mathcal{F}_2(y)$. Then by Proposition~\ref{l:complete}, the leaf $\mathcal{F}_1(y)$ is complete and flat. Then by real analyticity, all leaves of $\mathcal{F}_1$ are complete and flat and the theorem follows from \cite[Theorem~1]{PR}.

In the remaining part of the paper we prove the propositions. First in Section~\ref{s:prelim} we introduce and study the \emph{Fried metric}, which will play an important role in the proof. Then in Section~\ref{s:pfprop} we prove Proposition~\ref{p:Ri=0}, and in Section~\ref{ss:further}, Proposition~\ref{l:complete}.

\subsection{The Fried metric} 
\label{s:prelim}

Consider a (continuous) Riemannian metric $g_F$ on $\tilde{M}$ which is conformally equivalent to $g$ with the coefficient $\big(\tfrac{1}{d_\infty}\big)^2$:
\begin{equation*}
g_F:= \big(\tfrac{1}{d_\infty}\big)^2 g.
\end{equation*}
We will call $g_F$ the \emph{Fried metric}, because it is a generalization of a metric introduced by D.~Fried in \cite{Fried} (whose paper actually contains many ideas we use in the proofs), and will denote $d_F$ the distance function relative to that metric. (In this section we are not working with the Binet-Legendre metric, so that using the notation $g_F$ for the Fried metric should create no ambiguity).

It is easy to see that $g_F$ is $G$-invariant, which implies that it induces a Riemannian metric on $M=\tilde{M}/G$. We keep the notation $g_F$ for the projection of $g_F$ to $M$, and $d_F$ for the distance function on $(M, g_F)$. Note that $(M, g_F)$ is a complete $C^0$-Riemannian manifold.

\begin{Lemma}
\label{l:fried}
{\ }

\begin{enumerate}[{\rm (a)}]
 \item \label{it:d<df}
 For any $x, y \in \tilde{M}$ we have
 \begin{equation*}
 d(x,y) \le d_\infty(x) (e^{d_F(x,y)}-1).
 \end{equation*}

 \item \label{it:d>df}
 For any $x, y \in \tilde{M}$ with $d(x,y) < d_\infty(x)$ we have
 \begin{equation*}
 d(x,y) \ge d_\infty(x) (1-e^{-d_F(x,y)}).
 \end{equation*}
\end{enumerate}\end{Lemma}

\begin{proof}
\eqref{it:d<df} Let $d_F(x,y)=a$. Consider a minimizing geodesic $\gamma$ relative to $g_F$ connecting $x$ and $y$. 
We parameterize $\gamma$ by the arclength parameter $t$ relative to $g$. Let $\ell$ be its $g$-length and $\gamma(0)= x$ and $\gamma(\ell)= y$.
By construction,
\begin{equation} \label{eq1}
\int_{0}^\ell \tfrac{1}{d_\infty(\gamma(t))} dt \le a.
\end{equation}
By the triangle inequality we have
\begin{equation*}
\tfrac{1}{d_\infty(\gamma(t))} \ge \tfrac{1}{d_\infty(x)+ d(x, \gamma(t))}\ge \tfrac{1}{d_\infty(x)+ t} \, .
\end{equation*}
Combining this with \eqref{eq1} we obtain
\begin{equation*} 
\int_{0}^\ell \tfrac{1}{d_\infty(x)+ t} dt \le a,
\end{equation*}
so $\frac{d_\infty(x)+ \ell}{d_\infty(x)} \le e^a$. As $\ell\ge d(x,y)$ we get $\tfrac{1}{d_\infty(x)} d(x,y) \le \tfrac{\ell}{d_\infty(x)} \le e^a-1$, as required. 

\eqref{it:d>df} Let $d(x,y)=\ell < d_\infty(x)$. Consider a minimizing geodesic $\gamma$ relative to $g$ connecting $x$ and $y$. We parameterize $\gamma$ by the arclength parameter $t$ relative to $g$ so that $\gamma(0)= x$ and $\gamma(\ell)= y$. By the triangle inequality we have $ d_\infty(\gamma(t)) \ge d_\infty(x)-t$.
Then
\begin{equation*}
d_F(x,y) \le \int_{0}^\ell \tfrac{1}{d_\infty(\gamma(t))} dt \le \int_{0}^\ell \tfrac{1}{d_\infty(x)-t} dt = \ln\frac{d_\infty(x)}{d_\infty(x)-\ell} \, ,
\end{equation*}
and the claim follows.
\end{proof}

We will also need the following easy lemma which is close to \cite[Lemma~4.2,4.3]{BM}.
\begin{Lemma}
\label{l:rectangle}
Suppose at a point $x \in \tilde{M}$ the exponential map is defined on an open ball of radius $r>0$ in $T_x\tilde{M}$. Suppose $d(x,y) = \ell < r$ and $\gamma:[0,\ell] \to \tilde{M}$ is a shortest (arclength parameterized) geodesic with $\gamma(0)=x, \; \gamma(\ell)=y$. Let $v=\dot\gamma(0)$. Denote $v_i, \; i=1,2$, the projections of $v$ to the tangent spaces to the leaves of $\mathcal{F}_i$ passing through $x$ respectively. Then the geodesic  $\gamma_1(t):=\exp_x(tv_1)$  is defined for all $t \in [0,\ell]$ and lies on $\mathcal{F}_1$. Moreover, if $T(t)$ is the parallel vector field along $\gamma_1(t)$ with $T(0)=v_2$, then the map $\Phi: (t,s) \mapsto \exp_{\gamma_1(t)}(sT(t))$  is defined for all $(t,s) \in [0,\ell] \times [0,\ell]$, and its image is a flat totally geodesic immersed submanifold of $\tilde{M}$ with boundary (``rectangle"). In particular, $d(x, \Phi(t,s)) \le \sqrt{(t\|v_1\|)^2+(s\|v_2\|)^2}$.
\end{Lemma}
\begin{proof}
The first claim (that $\gamma_1(t)$ is defined for all $t \in [0,\ell]$  and lies  on $\mathcal{F}_1$) is obvious, as $\|v_1\| \le \ell <r$, and $\mathcal{F}_1$ is totally geodesic. The second claim is trivial  for  $v_1=0$ or $v_2=0$ since in these cases the rectangle degenerates to a naturally parameterized geodesic of length $\ell < r$ lying on one of the leaves.

\vspace{1ex}
Suppose $v_1, v_2 \ne 0$. Then $\Phi(0,s)$ is defined for $s \in [0, \ell]$. By compactness, there is an open neighborhood of the segment $\Phi(0,[0, \ell])$ which is isometric to the product of a small ball around $x$ on $\mathcal{F}_1$ and the segment $[0, \ell]$. It follows that $\Phi(t,[0, \ell])$ is defined for all $t \in [0, \ve)$ for some $\ve >0$, and its image is a flat totally geodesic rectangle in $(\tilde{M},g)$. Then for every $(t,s) \in [0, \ve) \times [0, \ell]$ we have $d(x, \Phi(t,s))^2 \le (t\|v_1\|)^2 + (s\|v_2\|)^2$. Let $\ve'$ be the supremum of such $\ve$. Suppose that $\ve' < \ell$ and let $s' \in (0, \ell]$ be the supremum of those $s$ for which $\Phi(\ve',s)$ is defined (so that ``$\Phi(\ve',s')$ lies on the metric boundary of $(\tilde{M},g)$"). But then the sequence of points $\Phi(\ve'-\frac1n,s')$ is a Cauchy sequence and $d(x, \Phi(\ve'-\frac1n,s')) \le (((\ve'-\frac1n)\|v_1\|)^2 + (s'\|v_2\|)^2)^{1/2} \le (\ve'\|v_1\|)^2 + (\ell\|v_2\|)^2)^{1/2} < \ell$, which contradicts the fact that $\exp_x$ is defined on the open ball of radius $r > \ell$.
\end{proof}

\subsection{Proof of Proposition~\ref{p:Ri=0}}
\label{s:pfprop}

Take an arbitrary point $x \in \tilde{M}$. Let $p\in M_\infty$ be such that $d(x,p)=d_\infty(x)$ (in $\overline{M}$) and let $\gamma:[0, d_\infty(x)] \to \overline{M}$ be a minimal $g$-geodesic connecting $x$ with $p$. Let $t \in [0, d_\infty(x)]$ be the $g$-natural parameter on $\gamma$ so that $\gamma(0)= x$ and $\gamma(d_\infty(x))= p$. Note that the $g$-exponential map at $x$ is defined on the ball of radius $d_\infty(x)$ in $T_x\tilde{M}$. Denote its image $B\subset \tilde{M}$; note that $p$ lies on the metric boundary of $B$.

Let $v:=\dot \gamma(0)$ be the initial velocity vector of $\gamma$ and let $\alpha \in [0, \frac{\pi}2]$ be the angle which $v$ makes with $\mathcal{F}_1$. As $\mathcal{F}_1$ is $g$-totally geodesic, the angle between $\gamma(t)$ and $\mathcal{F}_1$ remains constant. In particular, if $\alpha=0$ (respectively, $\alpha=\frac{\pi}2$), then $\gamma$ lies on $\mathcal{F}_1$ (respectively, on~$\mathcal{F}_2$).


\begin{Lemma}\label{l:tube} 
Suppose $\alpha \ne 0$. Then there exists $t_0 \in (0, d_\infty(x))$ and $\varepsilon > 0$ such that at all the points in the open $\ve$-neighborhood of the segment $\gamma(t_0, d_\infty(x))$ relative to $d_F$ we have $R_1=0$.
\end{Lemma}
\begin{proof}
We first note that choosing $\ve$ to be smaller than $\ln 2$ we obtain that our neighborhood entirely lies in $B$, by Lemma~\ref{l:fried}\eqref{it:d<df} and the triangle inequality.

Furthermore, as $\pi: \tilde{M} \to M$ is a $d_F$-Riemannian cover and as $M$ is compact, there exists $\delta > 0$ such that whenever $d_F(y,z) < \delta$ for $y,z \in \tilde{M}$, we have $d_F(\pi(y),\pi(z))= d_F(y,z)$ (take $\delta$ to be half the injectivity radius of $(M, d_F)$).

Now take $\ve=\min(\frac13 \ln 2, \frac12\delta)$. By way of contradiction, suppose that there exists an increasing sequence of points $t_i \in (0, d_\infty(x))$ converging to $d_\infty(x)$ such that the $\ve$-neighborhood of each of the points $\gamma(t_i)$ relative to $d_F$ contains a point with $R_1 \ne 0$. As $M$ is compact, we can assume (passing to a subsequence if necessary) that all the points $\pi(\gamma(t_i))$ lie in an open ball $B_{\ve} (z)$ centered at some $z \in M$ of radius $\ve$ relative to $(M,d_F)$. Lifting up to $\tilde{M}$ we obtain a sequence of points $z_i \in \tilde{M}$ such that $\gamma(t_i)$ lies in the open $d_F$-ball of radius $\ve$ centered at $z_i$. Then by the choice of $\ve$, the open $d_F$-ball of radius $2\ve$ centered at $z_i$ entirely lies in $B$ and contains the open $d_F$-balls of radius $\ve$ centered at $\gamma(t_i)$. Moreover, such balls are pairwise disjoint and the restriction of $\pi$ to each of them is a global $d_F$-isometry.

As $\pi(z_i)=z$, there exists a sequence $h_i$ of elements in $G$ such that $h_i(z_1)=z_i$. Note that every $h_i$ acts on $\tilde{M}$ as an isometry of $d_F$ and as a homothety of $d$ with the coefficient $k_i=d_\infty(z_i)/d_\infty(z_1)$. By Lemma~\ref{l:fried}\eqref{it:d<df} and by construction we have $d_\infty(z_i) \le d(\gamma(t_i),z_i)+d_\infty(\gamma(t_i)) \le d_\infty(\gamma(t_i)) e^\ve=(d_\infty(x)-t_i) e^\ve$. Passing to a subsequence if necessary we can assume that $k_i$ monotonically decrease to zero. Now by our assumption the $d_F$-ball of radius $2\ve$ centered at $z_1$ contains a point $w$ such that $R_1(w) = c > 0$. Then the point $w_i=h_i(w)$ lies in the $d_F$-ball of radius $2\ve$ centered at $z_i$ and we have $R_1(w_i) = k_i^{-2}c \to \infty$.

Note that by the triangle inequality $d_F(\gamma(t_i), w_i) \le d_F(\gamma(t_i), z_i)+d_F(z_i, w_i) < 3\ve$, so by Lemma~\ref{l:fried}\eqref{it:d<df} $d(\gamma(t_i), w_i) < d_\infty(\gamma(t_i))(e^{3\ve}-1) < d_\infty(\gamma(t_i))$ by the choice of $\ve$. Applying Lemma~\ref{l:rectangle} to the point $\gamma(t_i)$ and a shortest $g$-geodesic joining $\gamma(t_i)$ and $w_i$ we find that there is a geodesic $\Gamma_1$ of length less than $d_\infty(\gamma(t_i))$ with the endpoints $\gamma(t_i)$ and some $p_i$ lying on the leaf $\mathcal{F}_1(\gamma(t_i))$ and another geodesic joining the points $w_i$ and $p_i$ lying on the leaf of $\mathcal{F}_2$. We get $R_1(p_i) = R_1(w_i) = k_i^{-2}c$. Now applying Lemma~\ref{l:rectangle} to the point $x$ and the geodesic segment $\gamma([0,t_i])$ we find that there is a geodesic of length $d_\infty(x) \cos \alpha$ with the endpoints $x$ and some $q_i$ lying on the leaf $\mathcal{F}_1(x)$ and another geodesic $\Gamma_2$ joining the points $q_i$ and $\gamma(t_i)$ lying on the leaf of $\mathcal{F}_2$. Moreover, again by Lemma~\ref{l:rectangle}, for any point $y \in \Gamma_2$, we have $d(x,y) \le d (x, \gamma(t_i))$, so $d_\infty(y) \le d_\infty(\gamma(t_i))$ by the triangle inequality. It follows that for all $y \in \Gamma_2$, the map $\exp_y$ is defined on an open ball of radius $d_\infty(\gamma(t_i))$. Then by \cite[Lemma~4.3(i)]{BM} (the proof of which for our $\tilde{M}$ is identical to that for $M_0$) there is a geodesic lying in a leaf of $\mathcal{F}_2$ and joining the point $p_i$ with a point $u_i$ lying on the leaf $\mathcal{F}_1(q_i)$ ($=\mathcal{F}_1(x)$). Note that $R(u_i)=R_1(p_i) = R_1(w_i) = k_i^{-2}c$ and that there is a geodesic of $\mathcal{F}_1$ joining the points $q_i$ and $u_i$ whose length is equal to the length of $\Gamma_1$ (``the projection" of $\Gamma_1$ along $\Gamma_2$), so that $d(q_i, u_i) < d_\infty(\gamma(t_i)) = d_\infty(x) - t_i$. It follows that when $i$ tends to infinity, the points $u_i$ converge to the limit $\lim_{i \to \infty}q_i = \lim_{i \to \infty}\exp_x(t_i v_1)= \exp_x(d_\infty(x) v_1)$, which lies in $B$ as $\|v_1\|<1$. On the other hand, $R(u_i) = k_i^{-2}c \to \infty$, a contradiction.
\end{proof}

We continue with the proof of the proposition. Note that if $\alpha=0$, then the geodesic $\gamma$ lies on the leaf of $\mathcal{F}_1(x)$. It follows that $R_2$ is constant along $\gamma$, so the proof is finished by Lemma~\ref{l:tube}. Similar arguments work for $\alpha=\frac{\pi}{2}$. We therefore assume that $\alpha \in (0, \frac{\pi}{2})$, so that $\gamma$ does not lie on any leaf. As it now follows from Lemma~\ref{l:tube}, there is a point $t_0 \in [0, d_\infty(x))$ and a number $\ve > 0$ such that the $d_F$-tube of radius $\ve$ around the segment $\gamma([t_0, d_\infty(x)))$ is $g$-flat. Now for a point $\gamma(t), \; t \in [t_0, d_\infty(x))$ let $y \in \tilde{M}$ satisfy $d(\gamma(t),y) < (1-e^{-\ve}) d_\infty(\gamma(t))$. Then by Lemma~\ref{l:fried}\eqref{it:d>df} we have $d_F(\gamma(t),y) < \ve$. Therefore the open $g$-ball of radius $(1-e^{-\ve}) d_\infty(\gamma(t))$ centered at $\gamma(t)$ lies in the open $g_F$-ball of radius $\ve$ centered at $\gamma(t)$. It follows that the open $d_F$-tube of radius $\ve$ around the segment $\gamma([t_0, d_\infty(x)))$ contains the union of open $g$-balls $B_{(1-e^{-\ve}) d_\infty(\gamma(t))}(\gamma(t))$, and moreover, the metric $g$ on this union is flat. It follows that this union contains an open (solid) Euclidean cone of revolution $\mathcal{C}$ of height $d_\infty(x)-t_0$, with the apex at $p$ whose axes is the Euclidean segment $\gamma((t_0, d_\infty(x)))$ and with the angle $\beta=\arcsin(1-e^{-\ve})$ between the axis and the directrix. We also note that by Lemma~\ref{l:fried}\eqref{it:d<df}, all the points of $\mathcal{C}$ lie in an $d_F$-neighborhood of $\gamma((t_0, d_\infty(x)))$ of radius $\ve'$, where $e^{\ve'}-1 < 1-e^{\ve}$ (note that this implies $\ve' < \ve$).

Now, similar to the proof of Lemma~\ref{l:tube}, take an increasing sequence of points $t_i \in (t_0, d_\infty(x))$ such that the points $\pi(\gamma(t_i))$ converge to a certain point $z \in M$ relative to $(M, g_F)$ and additionally such that the ($d_F$-)unit tangent vectors $V_i$ at the points $\pi(\gamma(t_i))$ converge to a certain unit vector $V \in T_vM$ (in the topology of the $d_F$-unit tangent bundle of $M$). We can assume that $d_F(z, \pi(\gamma(t_i))) < \frac12 \ve'$. Lifting up to $\tilde{M}$ we obtain a sequence of points $z_i=h_i(z_1)$ (such that $\pi(z_i)=z$ and that $h_i \in G$) and a sequence of open $d_F$-balls
$B_i$ of radius $\frac12 \ve'$ centered at $z_i$ and containing $\gamma(t_i)$. Note that all the balls $B_i$ lie in $\mathcal{C}$. Moreover, as $g$ and $g_F$ are conformally equivalent, we obtain that the unit tangent vectors $\dot\gamma(t_i)$ converge (after parallel translation to some fixed point relative to the flat metric of $\mathcal{C}$) to a certain fixed vector. We can now additionally require that the angle between $\dot\gamma(t_i)$ and $\dot\gamma(t_j)$ is less than $\frac14 \beta$ (passing to a subsequence, if necessary). Now take $j \gg i$ and consider the image of the geodesic segment $\gamma([0,t_i])$ under the action of the element $h_jh_i^{-1}$. This element acts as an isometry of $(\tilde{M},d_F)$; it maps $z_i$ to $z_j$ and hence the ball $B_i$ onto the ball $B_j$, and therefore the point $\gamma(t_i)$ to a certain point $y_i$ in $B_j$. By construction, $y_i \in \mathcal{C}$ and moreover, the image of the segment $\gamma([0,t_i])$ under $h_jh_i^{-1}$ is a geodesic of $(\tilde{M}, g)$ which starts at $y_i$ with the tangent vector $-\dot\gamma(t_i)$ and of length $k_jk_i^{-1}t_i$, where $k_i$ is the homothety coefficient of $h_i$ (relative to $(\tilde{M}, g)$). Taking $j$ very large we can make this length arbitrarily small, and moreover, as the tangent vector to the $g$-geodesic $h_jh_i^{-1}\gamma([0,t_i])$ at $y_i$ makes the angle less than $\frac12 \beta$ with the axis of the cone $\mathcal{C}$, we obtain that the whole image $h_jh_i^{-1}\gamma([0,t_i])$ lies entirely in $\mathcal{C}$. It follows that the curvature of $(\tilde{M},g)$ at the point $h_jh_i^{-1}(x)=h_jh_i^{-1}\gamma(0)$ vanishes, hence it also does at $x$. \qed

\subsection{Proof of Proposition \ref{l:complete}}
\label{ss:further}


In our assumptions, the function $R_2$ has the same nonzero value at all the points of $\mathcal{F}_1(x)$. By Proposition~\ref{p:Ri=0} it follows that all the minimal geodesics joining the points of $\mathcal{F}_1(x)$ to the metric boundary $M_\infty$ lie in the leaves of $\mathcal{F}_2$.

Let $\tau:(a,b) \to \tilde{M}$ be a naturally parameterized ($g$-)geodesic segment lying on $\mathcal{F}_1(x)$ such that $m:=\inf(d_\infty(\tau(t)) \, : \, t \in (a,b))>0$. Denote $l(t):=d_\infty(\tau(t))$. Choose an arbitrary open subinterval $I \subset (a,b)$ whose length is strictly smaller than $m$. The union of all such intervals $I$ covers $(a,b)$.

Now take arbitrary $t_1, t_2 \in I$ and denote $x_i=\tau(t_i)$. For every unit vector $X_1 \in T_{x_1}\mathcal{F}_2$ consider the geodesic $\rho:[0,l(t_1)) \to \tilde{M}$ such that $\rho(0)=x_1, \, \rho'(0)=X_1$. The geodesic $\rho$ lies on the leaf $\mathcal{F}_2(x_1)$ and is well-defined, as the exponential map at $x_1$. Let $X_2 \in T_{x_2}\mathcal{F}_2$ be the unit vector obtained by parallel translation of $X_1$ along $\tau([t_1,t_2])$. By Lemma~\ref{l:rectangle} we find that the geodesic $\exp_{x_2}(sX_2)$ is defined for all $s < \sqrt{l(t_1)^2-(t_2-t_1)^2}$. As $X_1 \in T_{x_1}\mathcal{F}_2$ was arbitrary and as the shortest geodesic joining $x_2$ to $M_\infty$ lies in $\mathcal{F}_2(x_2)$, it follows that $l(t_2) \ge \sqrt{l(t_1)^2-(t_2-t_1)^2}$, so that $l(t_1)^2-l(t_2)^2 \le (t_2-t_1)^2$. Interchanging the roles of $t_1, t_2$ we obtain $|l(t_1)^2 - l(t_2)^2| \le (t_1-t_2)^2$. Subdividing the segment $(t_1,t_2)$ into $n$ equal subsegments we get by the triangle inequality that $|l(t_1)^2 - l(t_2)^2| \le \frac{1}{n}(t_1-t_2)^2$, for any $n \in \mathbb{N}$. It follows that the restriction of $l$ to $I$ is constant, hence $l$ is constant on the whole $(a,b)$. Therefore the restriction of $d_\infty$ to every geodesic segment $\tau(a,b)$ lying on leaf of $\mathcal{F}_1(x)$ such that $\inf(d_\infty(\tau(t)) \, : \, t \in (a,b))>0$ is constant, hence by continuity of $d_\infty$ and connectedness of $\mathcal{F}_1(x)$, the restriction of $d_\infty$ to $\mathcal{F}_1(x)$ is constant.

Furthermore, the fact that $\mathcal{F}_1(x)$ is flat directly follows from Proposition~\ref{p:Ri=0}. The fact that it is complete follows from the fact that the exponential map on $\mathcal{F}_1(x)$ is defined on the ball of radius $d_\infty(x) > 0$ at every point of $\mathcal{F}_1(x)$. \qed

\begin{remark*}
The above proof of Proposition~\ref{l:complete} completes the proof of Theorem~\ref{thm:main}. Note that neither of Propositions~\ref{p:Ri=0} and~\ref{l:complete} relies on the analyticity assumption, and one may expect that Theorem~\ref{thm:main} still holds if we drop it. Assuming that $\tilde M$ is just smooth, one can establish the following two facts: first, the easy fact that ``the limit" of minimal geodesics is a minimal geodesic, and second, that if there are more than one minimal geodesic starting at a given point, and one of them is tangent to a leaf of $\mathcal{F}_i$, then all the others also do. It then follows from Proposition~\ref{l:complete} that the closures of the points at which the minimal geodesics lie on $\mathcal{F}_1$ and on $\mathcal{F}_2$ are disjoint. So to prove the theorem in the smooth case, one needs to show that the following is impossible: we have two disjoint closed sets made of complete flat leaves of $\mathcal{F}_1$ and of complete flat leaves of $\mathcal{F}_2$, respectively, and an open flat domain in between.
\end{remark*}

\end{document}